\documentclass{amsart}
\usepackage[bookmarks=true,colorlinks,urlcolor=black,citecolor=black,linkcolor=black]{hyperref}
\usepackage{mathrsfs}
\usepackage{mathtools}
\usepackage{amsmath}
\usepackage{amssymb}
\usepackage{graphicx}
\usepackage{latexsym}

\numberwithin{equation}{section}
\newtheorem{theorem}{Theorem}

\theoremstyle{definition}

\theoremstyle{definition}

\graphicspath{{./Figures/}{./30baiIMO2016/}{C:/Users/hung/Dropbox/Hung/Figures/}}
\begin{document}
\title{Some properties of rectangle and a random point}
\author[Quang Hung Tran]{Quang Hung Tran}
\address{High School for Gifted Students\\ Vietnam National University\\182 Luong The Vinh\\Hanoi\\Vietnam}
\email{tranquanghung@hus.edu.vn}
\date{\today}
\maketitle
\date{\today}

\begin{abstract}We establish a relationship between the two important central lines of the triangle, the Euler line and the Brocard axis, in a configuration with an arbitrary rectangle and a random point. The classical Cartesian coordinate system method shows its strength in these theorems. Along with that, some related problems on rectangles and a random point are proposed with similar solutions using Cartesian coordinate system.
\end{abstract}
\maketitle

\section{Introduction}In the geometry of triangles, the Euler line \cite{4b} plays an important role and is almost the most classic concept in this field. The Euler line passes through the centroid, orthocenter, and circumcenter of the triangle. The symmedian point \cite{4c} of a triangle is the concurrent point of its symmedian lines, the Brocard axis \cite{4d} is the line connecting the circumcircle and the symmedian point of that triangle. The Brocard axis is also a central line that plays an important role in the geometry of triangle.

The coordinate method was invented by Ren\'e Descartes from the 17th century \cite{0,0a}, up to now, the classical coordinate method of Descartes is still one of the most important method of mathematics and geometry.

In this paper, we apply Cartesian coordinate system to solve an interesting theorem for an arbitrary rectangle and a random point in which two important central lines in the triangle are mentioned as follows. The idea of using the Cartesian coordinate system is also used in a similar way to solve the new problems we introduced in Section 3.

\section{Main theorem and proof}

\begin{theorem}\label{thm1}Let $ABCD$ be a rectangle with center $I$. Let $P$ be a random point in its plane.
	\begin{itemize}
	\item[1)] Euler line of triangles $PAB$ and $PCD$ meet at $Q$. Euler line of triangles $PBC$ and $PAD$ meet at $R$. Then, line $QR$ goes through center $I$.
	\item[2)] Brocard axis of triangles $PAB$ and $PCD$ meet at $M$. Brocard axis of triangles $PBC$ and $PAD$ meet at $N$. Then, line $MN$ goes through $P$.
	\end{itemize}
\end{theorem}

\begin{proof}We will use Cartesian coordinate for the solutions. Since $ABCD$ is an arbitrary rectangle and $P$ is a random point, we can take $P(0,0)$, $A(a,b)$, $B(c,b)$, $C(c,d)$, and $D(a,d)$ for all real numbers $a$, $b$, $c$, and $d$. 
	
We get perpendicular bisectors $d_a$, $d_b$, $d_c$, and $d_d$ of $PA$, $PB$, $PC$, and $PD$, respectively, are
\begin{equation}d_a:\ y = \frac{a^{2} + b^{2}}{2b}-\frac{a}{b}\cdot x,
\end{equation}
\begin{equation}d_b:\ y = \frac{b^{2} + c^{2}}{2b}-\frac{c}{b}\cdot x,
\end{equation}
\begin{equation}d_c:\ y =  \frac{c^{2} + d^{2}}{2d}-\frac{c}{d}\cdot x,
\end{equation}
\begin{equation}d_d:\ y = \frac{a^{2} + d^{2}}{2d}-\frac{a}{d}\cdot x.
\end{equation}
Thus circumcenters $O_a$, $O_b$, $O_c$, and $O_d$ of triangles $PAB$, $PBC$, $PCD$, and $PDA$, respectively, are
\begin{equation}O_a=d_a\cap d_b=\left(\frac{a+c}{2}, \frac{-ac + b^{2}}{2b} \right),
\end{equation}
\begin{equation}O_b=d_b\cap d_c=\left(\frac{-bd + c^{2}}{2c}, \frac{b+d}{2}\right),
\end{equation}
\begin{equation}O_c=d_c\cap d_d=\left(\frac{a+c}{2}, \frac{-ac + d^{2}}{2d} \right),
\end{equation}
\begin{equation}O_d=d_d\cap d_a=\left(\frac{-bd+ a^{2}}{2a}, \frac{b+d}{2}\right).
\end{equation}

\begin{figure}[htbp]
	\begin{center}\scalebox{0.7}{\includegraphics{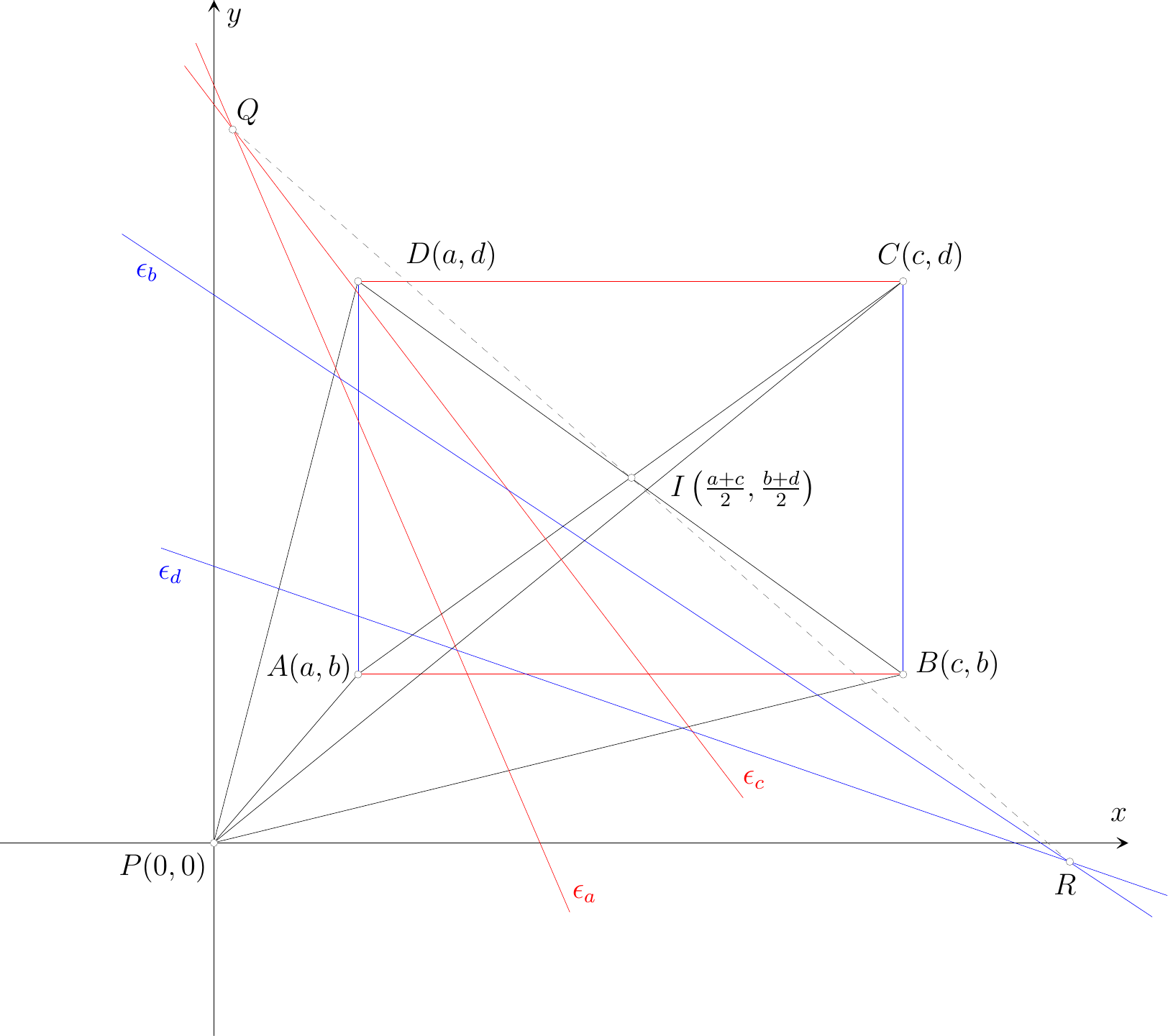}}\end{center}	
	\caption{Rectangle with Euler lines}
	\label{fig1}
\end{figure}

1) (See Figure \ref{fig1}). We have that centroid $G_a$, $G_b$, $G_c$, and $G_d$ of triangles $PAB$, $PBC$, $PCD$, and $PDA$, respectively, are
\begin{equation}G_a=\left(\frac{a + c}{3}, \frac{2b}{3} \right),
\end{equation}
\begin{equation}G_b=\left(\frac{2c}{3}, \frac{b + d}{3} \right),
\end{equation}
\begin{equation}G_c=\left(\frac{a + c}{3}, \frac{2d}{3} \right),
\end{equation}
\begin{equation}G_d=\left(\frac{2a}{3}, \frac{b + d}{3} \right).
\end{equation}
Therfore, Euler lines $\epsilon_a$, $\epsilon_b$, $\epsilon_c$, and $\epsilon_d$ of triangles $PAB$, $PBC$, $PCD$, and $PDA$, respectively, are
\begin{equation}\epsilon_a: \, y = \frac{b^{2} + ac}{b} - \frac{b^{2} + 3ac}{ab + bc}\cdot x,
\end{equation}
\begin{equation}\epsilon_b: \, y = \frac{bc^{2} + bd^{2} + b^{2}d + c^{2}d}{c^{2} + 3bd} - \frac{bc + cd}{c^{2} + 3bd}\cdot x,
\end{equation}
\begin{equation}\epsilon_c: \, y = \frac{d^{2} + ac}{d} - \frac{d^{2} + 3ac}{ad + cd}\cdot x,
\end{equation}
\begin{equation}\epsilon_d: \, y = \frac{bd^{2} + a^{2}b + a^{2}d + b^{2}d}{a^{2} + 3bd} - \frac{ab + ad}{a^{2} + 3bd}\cdot x.
\end{equation}
We get the intersections
\begin{equation}Q=\epsilon_a\cap\epsilon_c=\left(\frac{a^{2}c - abd + ac^{2} - bcd}{3ac - bd}, \frac{2abc + 2acd}{3ac - bd} \right),
\end{equation}
\begin{equation}R=\epsilon_b\cap\epsilon_d=\left(\frac{-2abd - 2bcd}{ac - 3bd}, \frac{abc + acd - b^{2}d - bd^{2}}{ac - 3bd} \right),
\end{equation}
and then the line connecting points $Q$ and $R$ is
\begin{equation}QR:\ y = -\frac{b+d}{a+c}x+b+d.
\end{equation}
Now it not hard to see that line $PQ$ goes through center $I\left(\frac{a+c}{2},\frac{b+d}{2}\right)$.

\begin{figure}[htbp]
	\begin{center}\scalebox{0.7}{\includegraphics{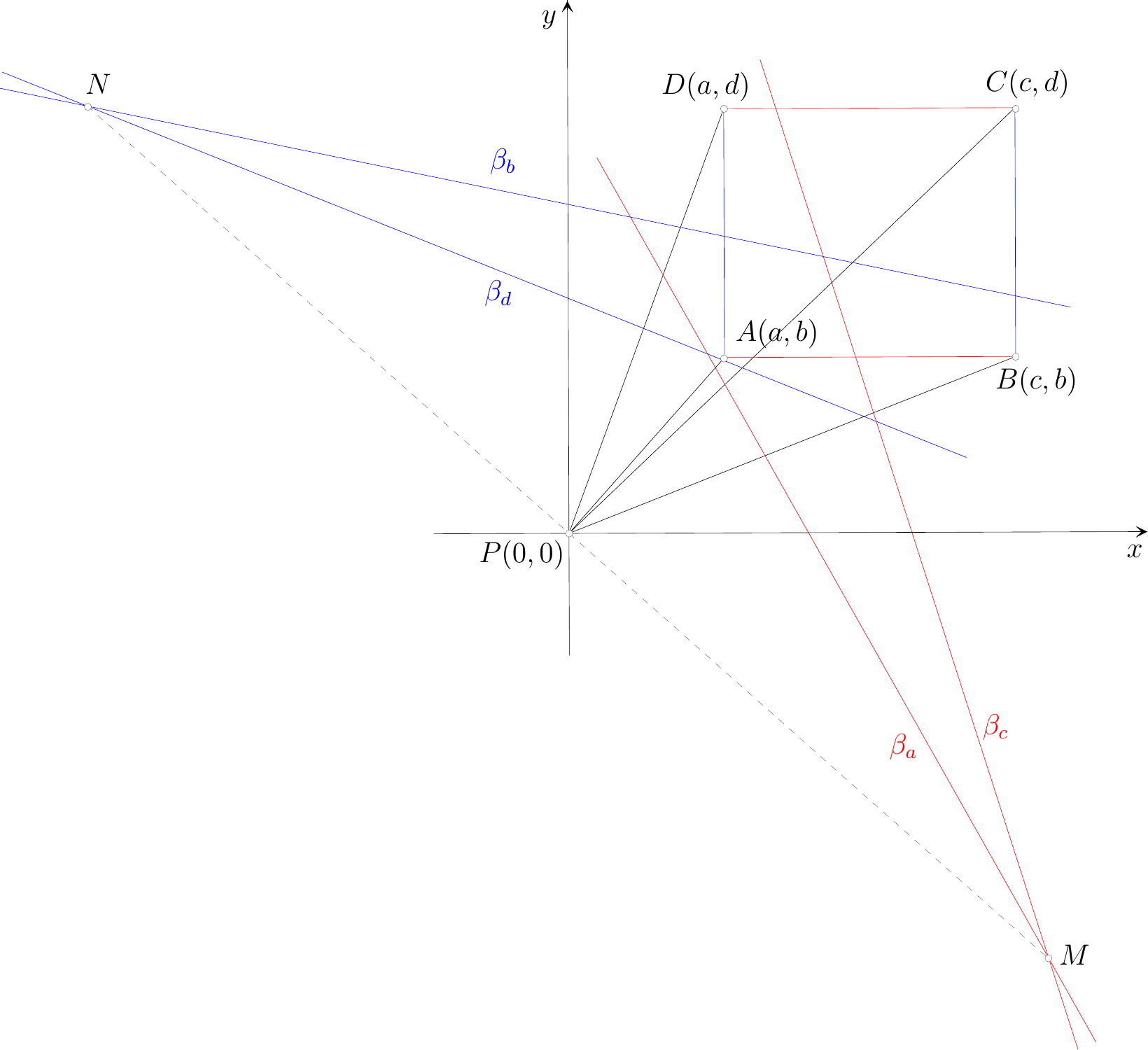}}\end{center}	
	\caption{Rectangle with Brocard axis}
	\label{fig2}
\end{figure}

2) (See Figure \ref{fig2}). Using the barycentric coordinates of symmedian point in \cite{4c}, we have that the coordinates of symmedian points $S_a$, $S_b$, $S_c$, and $S_d$ of triangles $PAB$, $PBC$, $PCD$, and $PDA$, respectively, are
\begin{equation}S_a=\frac{PA^2\cdot B+PB^2\cdot A+AB^2\cdot P}{PA^2+PB^2+AB^2}=\left(\frac{a^{2}c + ab^{2} + ac^{2} + b^{2}c}{2a^{2} - 2ac + 2b^{2} + 2c^{2}}, \frac{a^{2}b + 2b^{3} + bc^{2}}{2a^{2} - 2ac + 2b^{2} + 2c^{2}} \right),
\end{equation}
\begin{equation}S_b=\frac{PB^2\cdot C+PC^2\cdot B+BC^2\cdot P}{PB^2+PC^2+BC^2}=\left(\frac{b^{2}c + 2c^{3} + cd^{2}}{2b^{2} - 2bd + 2c^{2} + 2d^{2}}, \frac{b^{2}d + bc^{2} + bd^{2} + c^{2}d}{2b^{2} - 2bd + 2c^{2} + 2d^{2}} \right),
\end{equation}
\begin{equation}S_c=\frac{PC^2\cdot D+PD^2\cdot C+CD^2\cdot P}{PC^2+PD^2+CD^2}=\left(\frac{a^{2}c + ac^{2} + ad^{2} + cd^{2}}{2a^{2} - 2ac + 2c^{2} + 2d^{2}}, \frac{a^{2}d + c^{2}d + 2d^{3}}{2a^{2} - 2ac + 2c^{2} + 2d^{2}} \right),
\end{equation}
\begin{equation}S_d=\frac{PD^2\cdot A+PA^2\cdot D+DA^2\cdot P}{PD^2+PA^2+DA^2}=\left(\frac{2a^{3} + ab^{2} + ad^{2}}{2a^{2} + 2b^{2} - 2bd + 2d^{2}}, \frac{a^{2}b + a^{2}d + b^{2}d + bd^{2}}{2a^{2} + 2b^{2} - 2bd + 2d^{2}} \right).
\end{equation}
Therfore, Brocard axis $\beta_a$, $\beta_b$, $\beta_c$, and $\beta_d$ of triangles $PAB$, $PBC$, $PCD$, and $PDA$, respectively, are
\begin{equation}\beta_a: \, y = \frac{b^{4} + a^{2}b^{2} + a^{2}c^{2} + b^{2}c^{2}}{2bc^{2} + 2a^{2}b - 4abc} + \frac{-b^{4} - ac^{3} + a^{2}c^{2} - a^{3}c - 2ab^{2}c}{bc^{3} + a^{3}b - abc^{2} - a^{2}bc}\cdot x,
\end{equation}
\begin{equation}\beta_b: \, y = \frac{\begin{aligned}bc^{4} + b^{2}d^{3} + b^{3}c^{2} + b^{3}d^{2}+\\ + c^{2}d^{3}+ c^{4}d + bc^{2}d^{2} + b^{2}c^{2}d\end{aligned}}{2c^{4} + 2bd^{3} - 2b^{2}d^{2} + 2b^{3}d + 4bc^{2}d} + \frac{-cd^{3} - b^{3}c + bcd^{2} + b^{2}cd}{c^{4} + bd^{3} - b^{2}d^{2} + b^{3}d + 2bc^{2}d}\cdot x,
\end{equation}
\begin{equation}\beta_c: \, y = \frac{d^{4} + a^{2}c^{2} + a^{2}d^{2} + c^{2}d^{2}}{2a^{2}d + 2c^{2}d - 4acd} + \frac{-d^{4} - ac^{3} + a^{2}c^{2} - a^{3}c - 2acd^{2}}{a^{3}d + c^{3}d - ac^{2}d - a^{2}cd}\cdot x,
\end{equation}
\begin{equation}\beta_d: \, y = \frac{\begin{aligned}a^{2}b^{3} + a^{2}d^{3} + a^{4}b + a^{4}d+\\ + b^{2}d^{3} + b^{3}d^{2} + a^{2}bd^{2} + a^{2}b^{2}d\end{aligned}}{2a^{4} + 2bd^{3} - 2b^{2}d^{2} + 2b^{3}d + 4a^{2}bd} + \frac{-ab^{3} - ad^{3} + abd^{2} + ab^{2}d}{a^{4} + bd^{3} - b^{2}d^{2} + b^{3}d + 2a^{2}bd}\cdot x.
\end{equation}
We get the intersections
\begin{equation}M=\beta_a\cap\beta_c=\left(\frac{\begin{aligned}-a^{3}bd + a^{3}c^{2} - a^{2}bcd + a^{2}c^{3}-\\ - ab^{3}d - ab^{2}d^{2} - abc^{2}d - abd^{3}-\\ - b^{3}cd - b^{2}cd^{2} - bc^{3}d - bcd^{3}\end{aligned}}{\begin{aligned}2a^{3}c - 2a^{2}c^{2} - 4abcd+\\ + 2ac^{3} - 2b^{3}d - 2b^{2}d^{2} - 2bd^{3}\end{aligned}}, \frac{\begin{aligned}a^{3}bc + a^{3}cd + a^{2}bc^{2} + a^{2}c^{2}d+\\ + ab^{3}c + ab^{2}cd+ abc^{3}+ abcd^{2} +\\+ ac^{3}d + acd^{3} - b^{3}d^{2} - b^{2}d^{3}\end{aligned}}{\begin{aligned}2a^{3}c - 2a^{2}c^{2} - 4abcd+\\ + 2ac^{3} - 2b^{3}d - 2b^{2}d^{2} - 2bd^{3}\end{aligned}} \right),
\end{equation}
\begin{equation}N=\beta_b\cap\beta_d=\left(\frac{\begin{aligned}-a^{3}bd + a^{3}c^{2} - a^{2}bcd + a^{2}c^{3}-\\ - ab^{3}d - ab^{2}d^{2}-abc^{2}d - abd^{3}-\\ - b^{3}cd - b^{2}cd^{2} - bc^{3}d - bcd^{3}\end{aligned}}{\begin{aligned}2a^{3}c + 2a^{2}c^{2} + 4abcd+\\ + 2ac^{3} - 2b^{3}d + 2b^{2}d^{2} - 2bd^{3}\end{aligned}}, \frac{\begin{aligned}a^{3}bc + a^{3}cd + a^{2}bc^{2} + a^{2}c^{2}d+\\ + ab^{3}c + ab^{2}cd +abc^{3} + abcd^{2}+\\ + ac^{3}d + acd^{3} - b^{3}d^{2} - b^{2}d^{3}\end{aligned}}{\begin{aligned}2a^{3}c + 2a^{2}c^{2} + 4abcd+\\ + 2ac^{3} - 2b^{3}d + 2b^{2}d^{2} - 2bd^{3}\end{aligned}} \right),
\end{equation}
and then the line connecting points $M$ and $N$ is
\begin{equation}MN:\ y = \frac{\begin{aligned}b^{2}d^{3} + b^{3}d^{2} - abc^{3} - acd^{3} - ab^{3}c - ac^{3}d-\\ - a^{2}bc^{2} - a^{2}c^{2}d - a^{3}bc - a^{3}cd - abcd^{2} - ab^{2}cd\end{aligned}}{\begin{aligned}-a^{2}c^{3} - a^{3}c^{2} + abd^{3} + ab^{2}d^{2} + ab^{3}d + bcd^{3}+\\ + bc^{3}d + a^{3}bd + b^{2}cd^{2} + b^{3}cd + abc^{2}d + a^{2}bcd\end{aligned}}\cdot x.
\end{equation}
It is easy to see that $MN$ goes through $P(0,0)$. This completes the proof.
\end{proof}
\newpage

\section{Some others theorems on rectangle and a random point}

We introduce some more theorems on rectangle and a random point, all of them can be solved by Cartesian coordinate as we work above.

\begin{theorem}[Generalization of Theorem \ref{thm1}]Let $A_1B_1C_1D_1$ and $A_2B_2C_2D_2$ be two rectangles with the same center $I$. Let $P$ be a random point in its plane. Euler lines of triangles $PA_1B_1$ and $PC_1D_1$ meet at $Q$. Euler lines of triangles $PA_2D_2$ and $PB_2C_2$ meet at $R$. Then, three points $Q$, $R$, and $I$ are collinear (See Figure \ref{fig4a}).
\end{theorem}

\begin{figure}[htbp]
	\begin{center}\scalebox{0.7}{\includegraphics{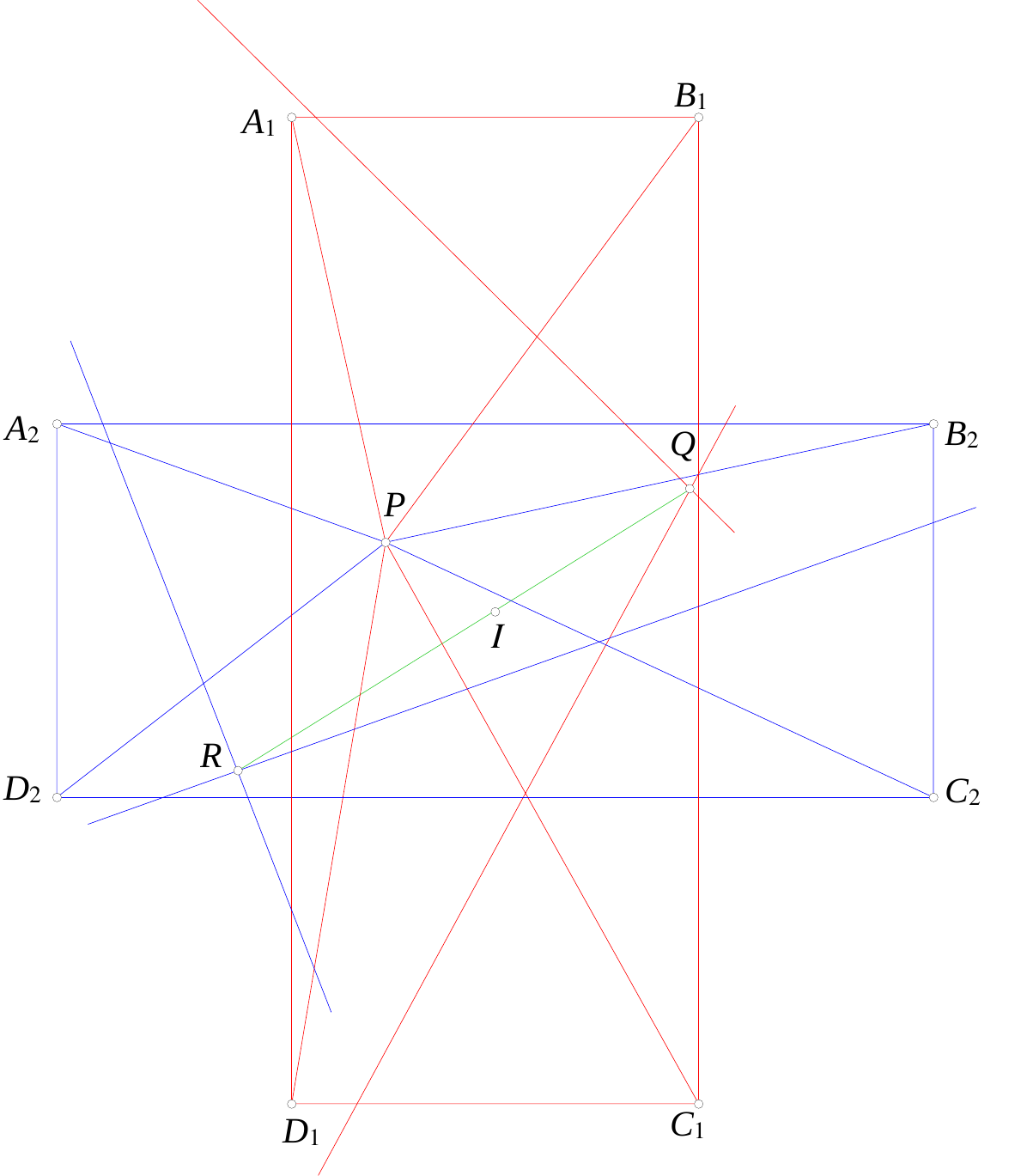}}\end{center}
	\caption{}
	\label{fig4a}
\end{figure}
\newpage

\begin{theorem}Let $ABCD$ be a rectangle with center $I$. Let $P$ be a random point in its plane. Let $H_a$, $H_b$, $H_c$, and $H_d$ be the orthocenters of triangles $PAB$, $PBC$, $PCD$, and $PDA$, respectively. Let $Q$ and $R$ be the midpoints of $H_aH_c$ and $H_bH_d$, respectively. Then, three points $Q$, $R$, and $I$ are collinear (See Figure \ref{fig4}).
\end{theorem}

\begin{figure}[htbp]
	\begin{center}\scalebox{0.7}{\includegraphics{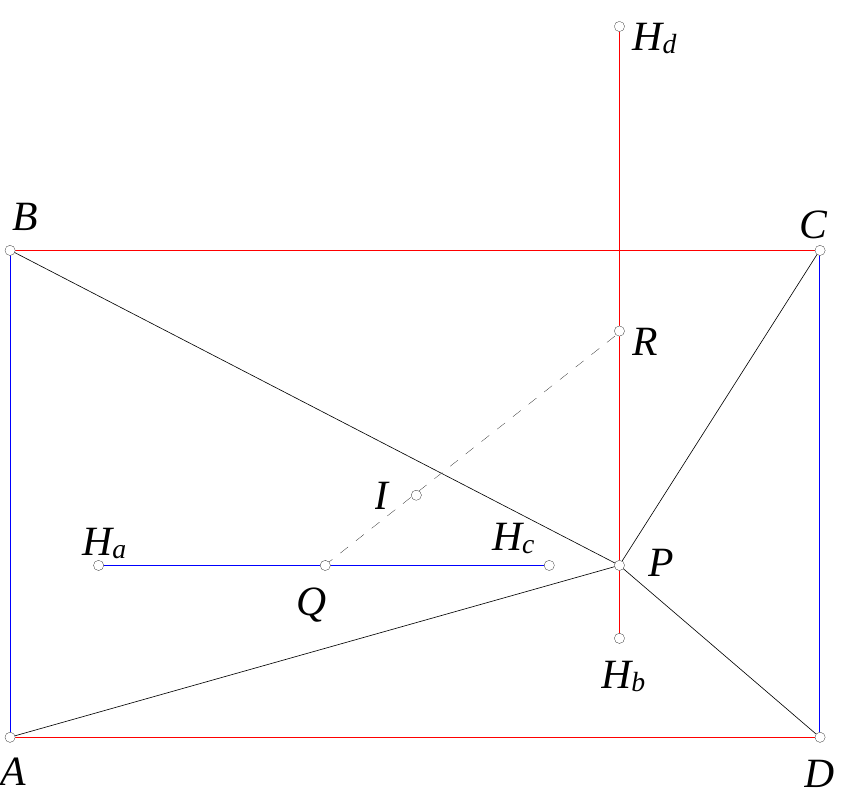}}\end{center}
	\caption{}
	\label{fig4}
\end{figure}
\newpage
\begin{theorem}Let $ABCD$ be a rectangle with center $I$. Let $P$ be a random point in its plane. Let $O_a$, $O_b$, $O_c$, and $O_d$ be the circumcenters of triangles $PAB$, $PBC$, $PCD$, and $PDA$, respectively. Let $Q$ and $R$ be the midpoints of $O_aO_c$ and $O_bO_d$, respectively. Let $S$ be the midpoint of $QR$. Circumcircles of triangles $PO_aO_c$ and $PO_bO_d$ meets again at $T$. Then, three points $S$, $T$, and $I$ are collinear (See Figure \ref{fig5}).
\end{theorem}

\begin{figure}[htbp]
	\begin{center}\scalebox{0.7}{\includegraphics{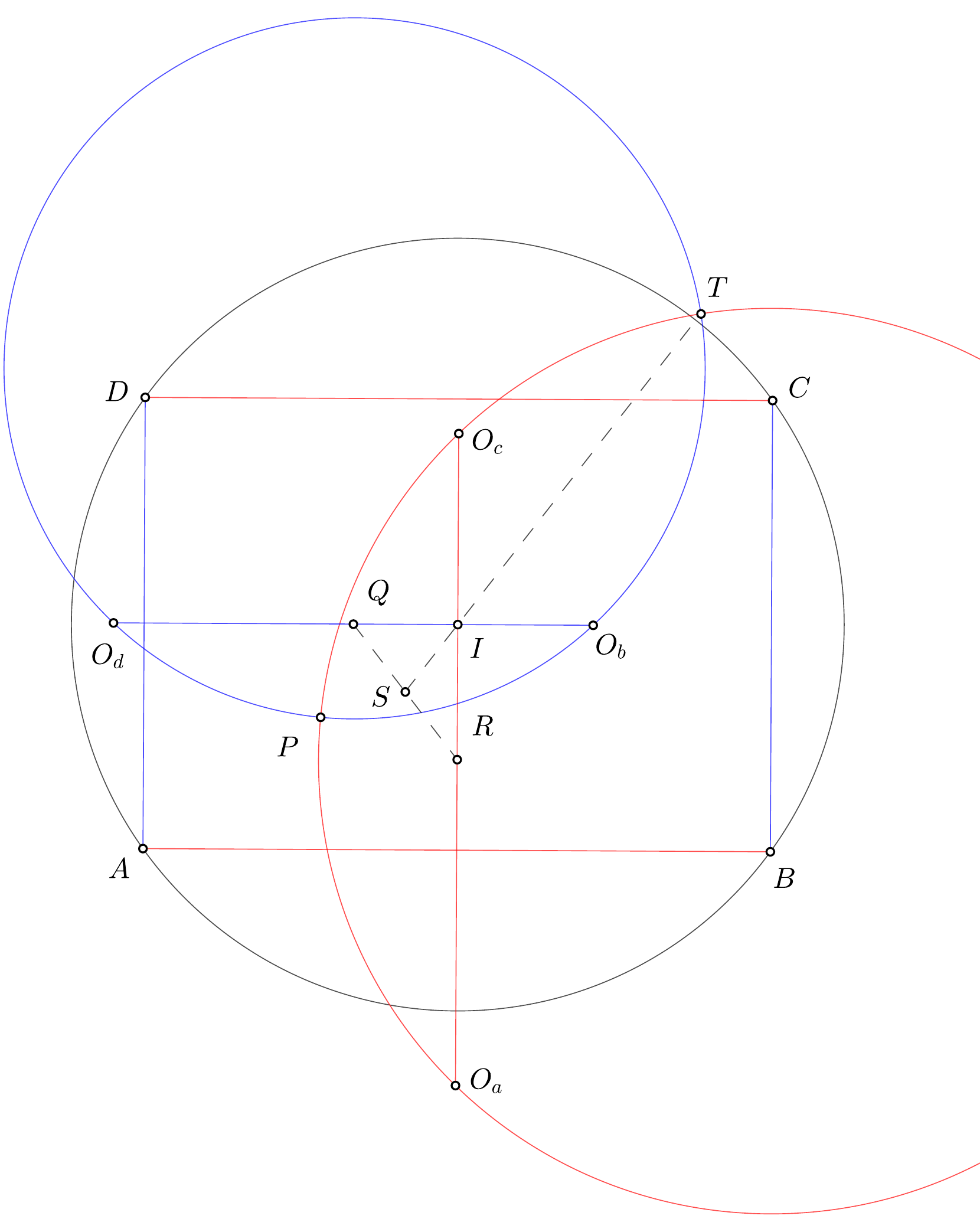}}\end{center}
	\caption{}
	\label{fig5}
\end{figure}
\newpage
\begin{theorem}Let $ABCD$ be a rectangle with center $I$. Let $P$ be a random point in its plane. Let $N_a$, $N_b$, $N_c$, and $N_d$ be the nine-point centers of triangles $PAB$, $PBC$, $PCD$, and $PDA$, respectively. Let $M$ and $N$ be the midpoints of $N_aN_c$ and $N_bN_d$, respectively. Let $Q$ be intersection of perpendicular bisectors of $N_aN_c$ and $N_bN_d$. Then, lines $MN$ and $IQ$ are parallel (See Figure \ref{fig6}).
\end{theorem}
\begin{figure}[htbp]
	\begin{center}\scalebox{0.7}{\includegraphics{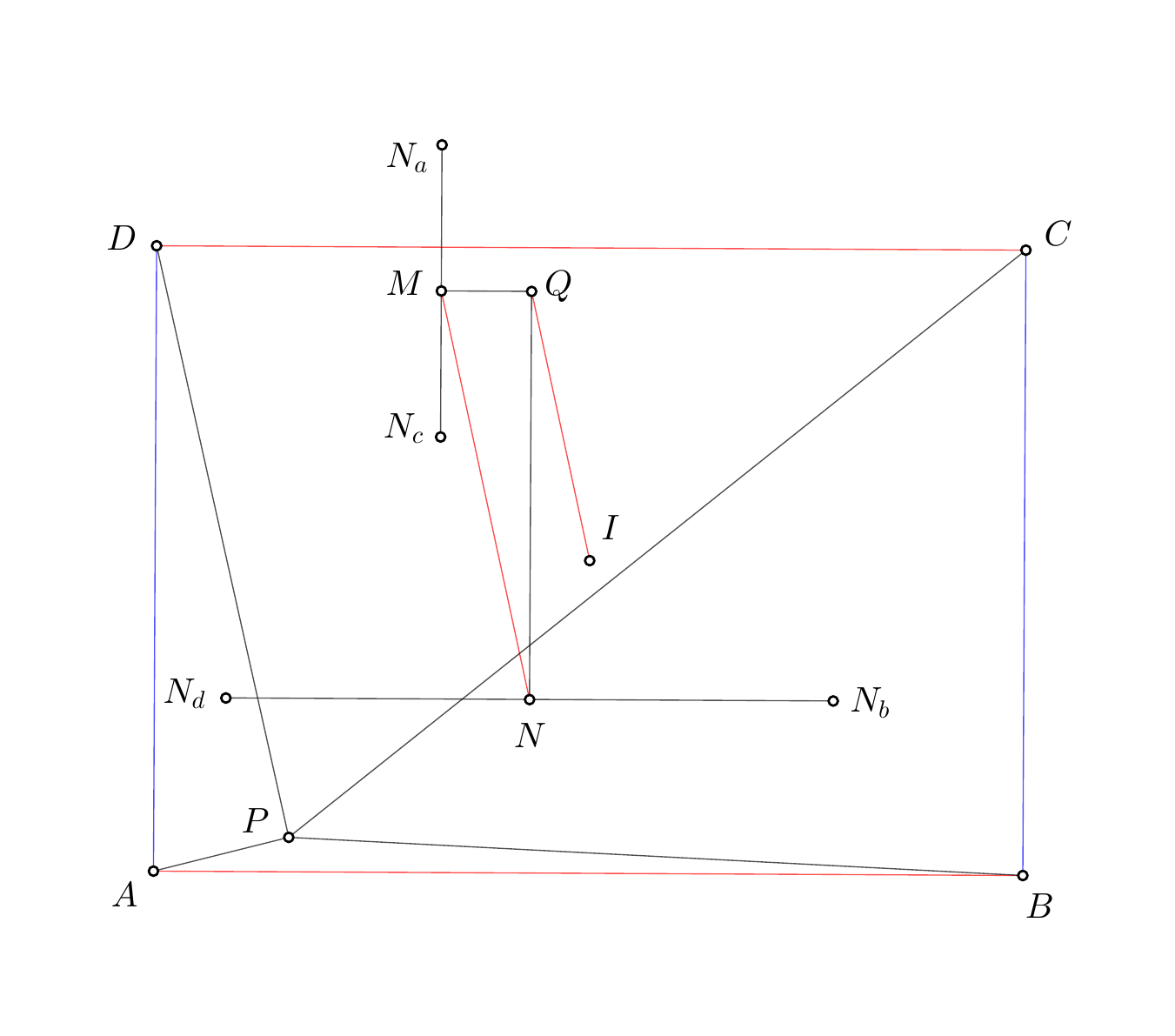}}\end{center}
	\caption{}
	\label{fig6}
\end{figure}

\newpage
\begin{theorem}Let $ABCD$ be a rectangle with center $I$. Let $P$ be a random point in its plane. Let $P_a$, $P_b$, $P_c$, and $P_d$ be the isogonal conjugate of $I$ with respect to triangles $PAB$, $PBC$, $PCD$, and $PDA$, respectively. Then, line $IP$ bisects the segments $P_aP_c$ and $P_bP_d$ (See Figure \ref{fig7}).
\end{theorem}
\begin{figure}[htbp]
	\begin{center}\scalebox{0.7}{\includegraphics{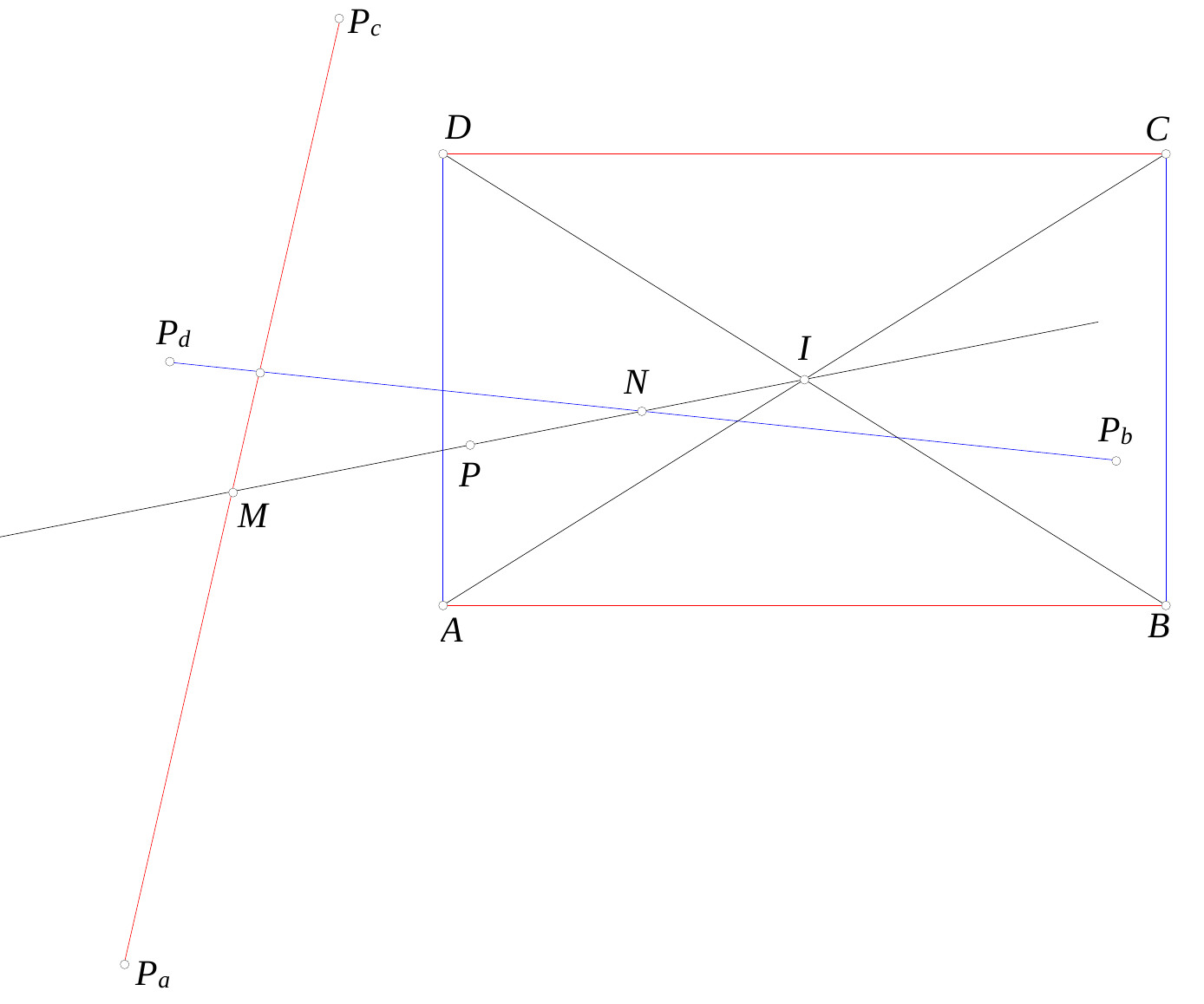}}\end{center}
	\caption{}
	\label{fig7}
\end{figure}
\newpage
\begin{theorem}Let $ABCD$ be a rectangle with center $I$. Let $P$ be a random point in its plane. Let $P_a$, $P_b$, $P_c$, and $P_d$ be the reflections of $P$ in the lines $AB$, $BC$, $CD$, and $DA$, respectively. Euler line of triangles $P_aAB$ and $P_cCD$ meet at $Q$. Euler line of triangles $P_bBC$ and $P_dDA$ meet at $R$. Then, line $IP$ bisects the segment $QR$ (See Figure \ref{fig8}).
\end{theorem}
\begin{figure}[htbp]
	\begin{center}\scalebox{0.7}{\includegraphics{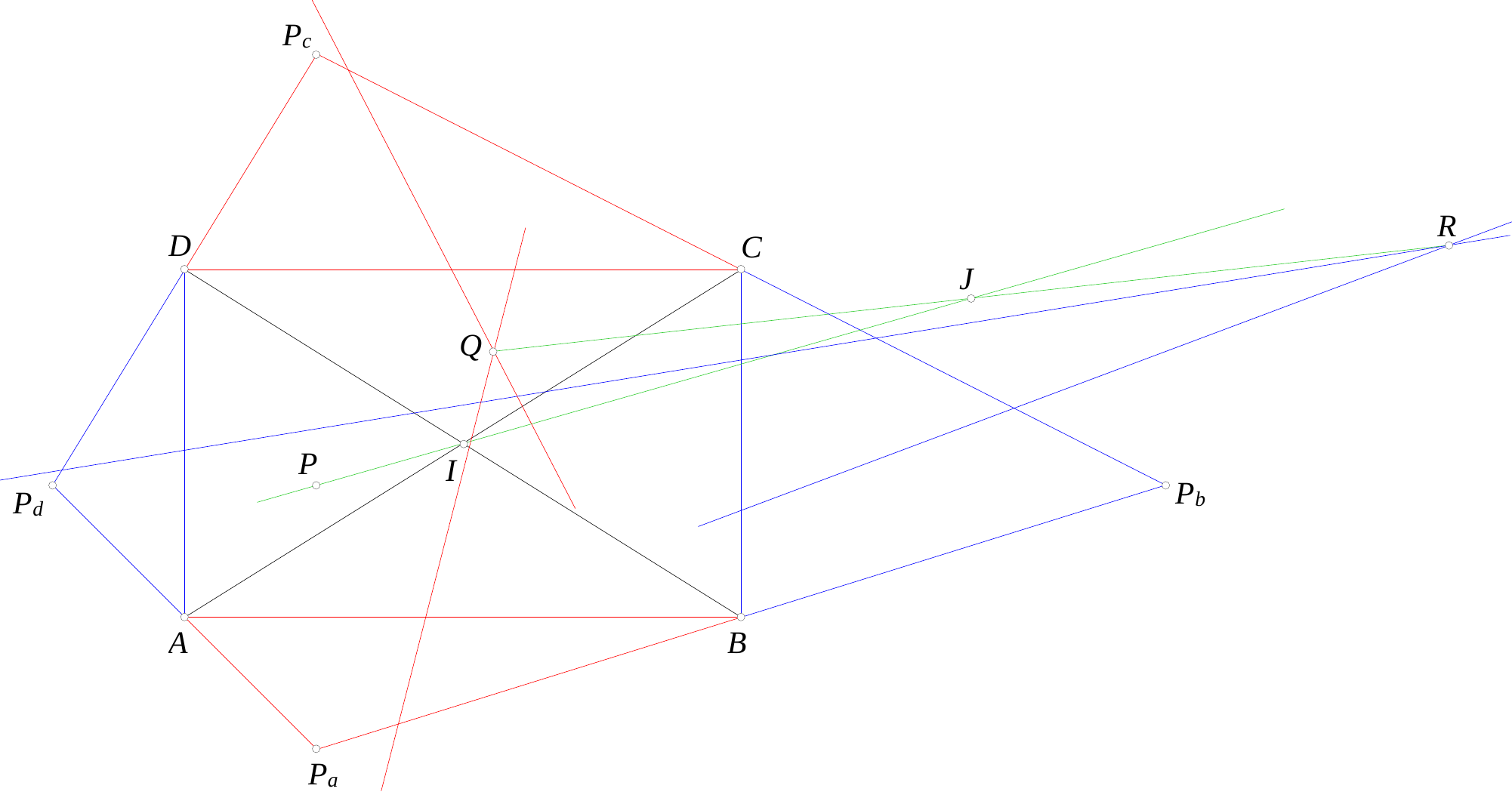}}\end{center}
	\caption{}
	\label{fig8}
\end{figure}
\newpage
\begin{theorem}Let $ABCD$ be a rectangle with center $I$. Let $P$ be a random point in its plane. Let $P_a$, $P_b$, $P_c$, and $P_d$ be the reflections of $P$ in the midpoints of sides $AB$, $BC$, $CD$, and $DA$, respectively. Euler line of triangles $P_aAB$ and $P_cCD$ meet at $Q$. Euler line of triangles $P_bBC$ and $P_dDA$ meet at $R$. Then, reflection of $P$ in $I$ lies on line $QR$ (See Figure \ref{fig9}).
\end{theorem}
\begin{figure}[htbp]
	\begin{center}\scalebox{0.7}{\includegraphics{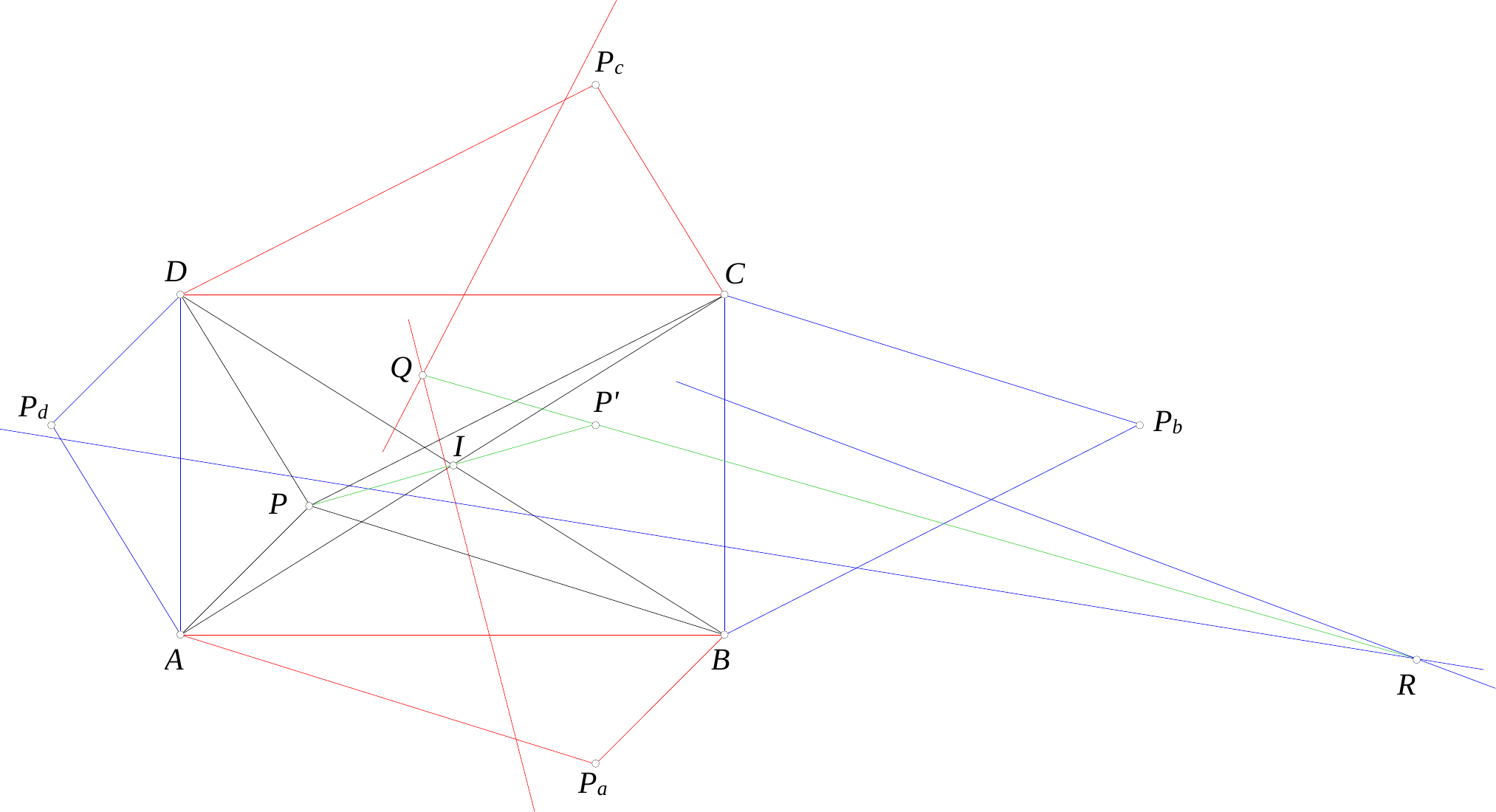}}\end{center}
	\caption{}
	\label{fig9}
\end{figure}
\newpage
\begin{theorem}Let $ABCD$ be a rectangle with center $I$. Let $P$ be a random point in its plane. Let $P_a$, $P_b$, $P_c$, $P_d$, $P_{ac}$, and $P_{bd}$ be the reflections of $P$ in the lines $AB$, $BC$, $CD$, $DA$, $AC$, and $BD$, respectively. Orthogonal projections of $P_{ac}$ on sides of quadrilateral $P_aP_bP_cP_d$ form a quadrilateral which has two diagonals meet at $Q$. Orthogonal projections of $P_{bd}$ on sides of quadrilateral $P_aP_bP_cP_d$ form a quadrilateral which has two diagonals meet at $R$ (See Figure \ref{fig10}). Then,
\begin{itemize}
\item[i)] Two points $P_{ac}$ and $P_{bd}$ are isogonal conjugate with respect to quadrilateral $P_aP_bP_cP_d$.
\item[ii)] There point $P$, $Q$, and $R$ are collinear.
\item[iii)] Two lines $QP_{ac}$ and $RP_{bd}$ are parallel.
\end{itemize}
\end{theorem}
\begin{figure}[htbp]
	\begin{center}\scalebox{0.7}{\includegraphics{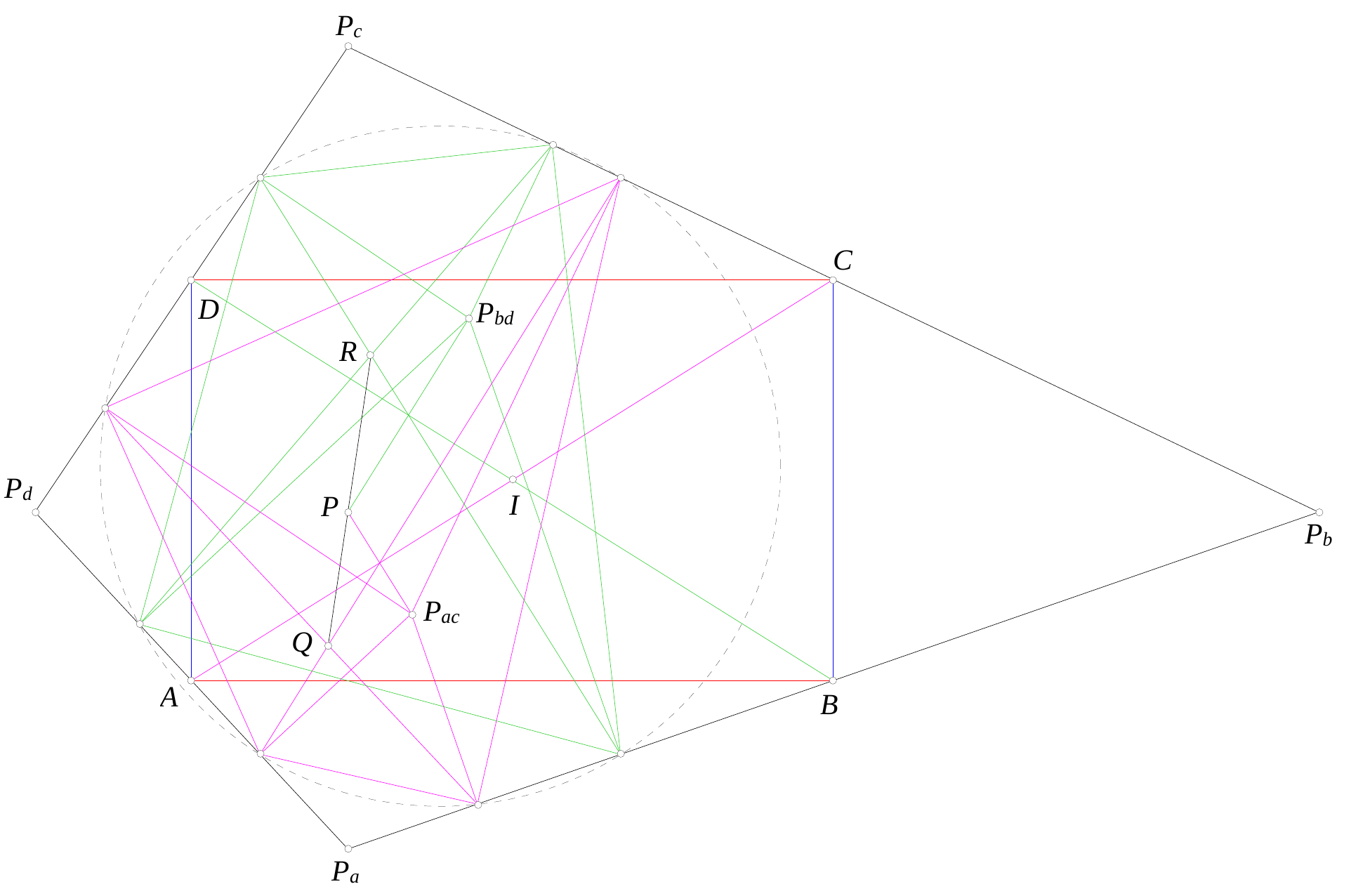}}\end{center}
	\caption{}
	\label{fig10}
\end{figure}

\end{document}